\documentclass{dmd2024}

\usepackage{graphicx,amsthm,amssymb,amsmath,authblk}

\usepackage{hyperref}

\newtheorem{theorem}{Theorem}
\newtheorem{lemma}[theorem]{Lemma}
\newtheorem{corollary}[theorem]{Corollary}

\newtheorem{definition}[theorem]{Definition}

\theoremstyle{definition}

\newcommand{\width}{\operatorname{width}}

\newcommand{\R}{\mathbb{R}}
\newcommand{\N}{\mathbb{N}}
\newcommand{\Sph}{\mathbb{S}}

\DeclareMathOperator{\interior}{inn}

\newcommand\doi[1]{\href{http://dx.doi.org/#1}{\texttt{doi:#1}}}

\renewcommand\epsilon{\varepsilon}

\usepackage[disable,colorinlistoftodos]{todonotes}
\newcommand{\pacoc}[1]{\todo[size=\tiny,color=pink!30]{#1  \hfill --- L.}}

\newcommand{\pacos}[1]{\todo[size=\tiny,color=green!30]{#1  \hfill --- P.}}

\title{The algorithmic Fried Potato Problem in two dimensions%
\thanks{This research is supported by Grants PID2019-106188GB-I00 and PID2022-137283NB-C21 of MCIN/AEI/10.13039/501100011033 / FEDER, UE and by project CLaPPo (21.SI03.64658) of Universidad de Cantabria and Banco Santander.
}}

\author[1]{Francisco Criado\thanks{Email: francisco.criado@cunef.edu}}
\author[2]{Francisco Santos\thanks{Email: francisco.santos@unican.es}}
\affil[1]{Departamento de Matem\'aticas, CUNEF Universidad, Madrid, Spain}
\affil[2]{Departamento de Matem\'aticas, Estad\'istica y Computaci\'on, Universidad de Cantabria, 39005 Santander, Spain}

\begin{document}

\maketitle

\begin{abstract}
  Conway's Fried Potato Problem seeks to determine the best way to cut a convex body in $n$ parts by $n-1$ hyperplane cuts (with the restriction that the $i$-th cut only divides in two one of the parts obtained so far), in a way as to minimize the maxuimum of the inradii of the parts. It was shown by Bezdek and Bezdek that the solution is always attained by $n-1$ parallel cuts.
  But the algorithmic problem of finding the best direction for these parallel cuts remains. 
  
  In this note we show that for a convex $m$-gon $P$, this direction (and hence the cuts themselves) can be found in time $O(m\log^4 m)$, which improves on a quadratic algorithm proposed by Ca\~nete-Fern\'andez-M\'arquez (DMD 2022).
  Our algorithm first preprocesses what we call the \emph{dome} (closely related to the medial axis) of $P$ using a variant of the Dobkin-Kirkpatrick hierarchy, so that linear programs in the dome and in slices of it can be solved in polylogarithmic time.
\end{abstract}

\section{From fried potatoes to baker's potatoes}

Conway's fried potato problem is stated in \cite{unsolved} (problem C1) as follows: ``In order to fry it as expeditiously as possible Conway wishes to slice a given convex potato into $n$ pieces by $n - 1$ successive plane cuts (just one piece being divided by each cut) so as to minimize the greatest inradius of the pieces.''

The problem was solved by A. Bezdek and K. Bezdek~\cite{Bezdek^2} who showed that, no matter what convex potato you start with, the best solution is to cut it with $n-1$ parallel and equally spaced hyperplanes. Let us formalize this a little bit:

\begin{definition}
  Let $C\subseteq \R^d$ be a convex body (that is, a compact convex subset with nonempty interior). 
  \begin{enumerate}

\item   The \emph{directional width} of  $C\subseteq \R^d$ in a direction $v\in\Sph^{d-1}$ is the distance between two parallel supporting hyperplanes of $C$ with normal vector $v$:
  \[ \width_v (C) = \max_{x \in C} v^T x - \min_{x\in C} v^T x. \]
  The \emph{width} of $C$ is its minimum directional width:
  \[ \width (C) = \min_{v \in \Sph^{d-1}} \width_v (C). \]

\item 
  The \emph{inner parallel body} of $C$ at a distance $\rho\ge 0$~\cite[p.~134]{Schneider}
  is the set of points of $C$ that are centers of balls of radius $\rho$ contained in $C$.
  \[ \interior_{\rho} (C) = \{x \in C : B(x, \rho) \subseteq C \}. \]
The $\rho$-rounded body $C^{\rho}$ is the union of all closed $\rho$-balls contained in $C$:
  \[ C^{\rho} = \bigcup_{x\in\interior_{\rho} (C)}   B(x, \rho). \]

\item The inradius $I(C)$ of $C$ is the maximum radius of a ball contained in $C$. Equivalently, it is the maximum $\rho$ for which $\interior_{\rho} (C)\ne \emptyset$.
\end{enumerate}
\end{definition}

Observe that 
$C^{\rho} = \interior_{\rho} (C) + B(0,\rho)$. Also if $C= \{x\in\R^d : Ax\leq b\}$ is a polyhedron with $||A_i|| = 1$ for each $i$, then $\interior_{\rho} (C) = \{x\in \R^d : Ax\leq b-\rho\}$, where $b-\rho$ is shorthand for $(b_1-\rho, \dots, b_m-\rho)$.

The statement and solution of Conways's fried potato problem can now be stated as follows:

\begin{theorem}[Bezdek-Bezdek~\cite{Bezdek^2}]
\label{thm:Bezdek}
\pacos{Check if this is stated there. If not, refer also to Cañete et al.}
  Let $C$ be a convex body in $\R^d$ and $n\in \N$. Let $P$ be a division of $C$ into $n$ subsets $C_1,\dots, C_n$ given by $n-1$ successive hyperplane cuts. These cuts of $P$ do not extend beyond previously made cuts, therefore $(n-1)$ cuts produce $n$ pieces.
  
  Then,
  \[ \max_{i\in [n] } I(C_i) \geq \rho, \]
  where $\rho>0$ is the unique number satisfying
  \begin{equation}
  \label{eq:rounded}
   \width(C^{\rho})=2n\rho.
\end{equation}

  Furthermore, equality holds for the division of $C$ given by $n-1$ parallel and equally spaced hyperplanes normal to the direction attaining $\width(C^{\rho})$.
 \end{theorem}

The solution to the fried potato problem raises the algorithmic question of how to find $\rho$, $v$ and the cuts in the statement. We suggest calling this the \emph{baker's potato problem}.\footnote{Baker's potatoes (\emph{pommes boulang\`ere} in French and \emph{patatas panaderas} in Spanish) are potatoes cut in parallel slices of 2--3 mm. and cooked in the oven.
}

Clearly, the difficult part is to find $\rho$ and the direction $v\in \Sph^{d-1}$ such that $\width_v(C^{\rho})=\width(C^{\rho})$. 
Ca\~nete, Fern\'andez and M\'arquez~\cite{CFM-dmd, CFM-journal} have proposed a quadratic algorithm to do this for a convex polygon in the plane. We here describe  a quasi-linear one:

\begin{theorem}
\label{thm:main}
  Let $P= \{x\in \R^2 : Ax\leq b \}$ be a polygon with non-empty interior, where $A\in \R_{m\times 2}$ and $b\in \R^m$. 
 We can compute the $\rho$ of Theorem~\ref{thm:Bezdek}
and a direction $v \in \Sph^2$ satisfying $\width_v (P^{\rho}) = 2n\rho$ in $O(m \log^4 m)$ time.
\end{theorem}

\section{The Dobkin-Kirkpatrick hyerarchy}

Equation~\ref{eq:rounded} suggests to formalize the Baker's potato problem adding one dimension to it. If, for a given convex body $C\in \R^d$, we define
\[
\overline C= \{(x,t)\in \R^d\times [0,\infty): x \in C^t\} \subset \R^{d+1}\},
\]
the problem to solve is to find the $\rho$ such that $\width(\overline C \cap \{t=\rho\}) = 2n\rho$.

We solve this using  the Dobkin-Kirkpartrick hierarchy, which allows to do linear programming queries in a $3$-dimensional polytope in logarithmic type per query. The classical version (which we do not use but state for completeness) is the following statement in which an \emph{extreme-point query} in a set $S$ of $m$ points has as input a linear functional $c\in \R^3$ and as output the point $p$ (or one of the points) of $S$ maximizing $c^T p$.

  \begin{theorem}[Dobkin-Kirkpatrick Hierarchy~\cite{DK1,DK2}, see also \protect{\cite[ Theorem 7.10.4]{orourke}}]
    After $O(m\log^2 m)$ time and space preprocessing, extreme-point queries in $3$ dimensions can be solved in $O(\log{m})$ time each.
  \end{theorem}

The version we need works in the dual. In what follows we will assume that the facet hyperplanes in our polytopes are generic, that is, no $d+1$ of them have a common point. This implies the polytopes to be simple and is not a  loss of generality since it can be achieved by a symbolic perturbation of the input matrix. 

\begin{definition}
\label{def:DK}
  Let $A\in \R^{m\times d}$, $b\in \R^m$ be the half-space description of a  polytope $P$ in $\R^d$. We call \emph{Dobkin-Kirkpatrick hierarchy} on $P$ a data structure consisting of:
 \begin{enumerate}
 \item The face poset of $P$, in which each face $F$ of codimension $k$ is represented by the subset of size $k$ of $[m]$ consisting of facets containing $F$.
 \item A stratification of the set $[m]$ as
 \[
 [m] = I_0 \supset I_1 \supset \dots \supset I_k
 \]
with the property that the facets labelled by each $I_l\setminus I_{l+1}$ are independent (i.e., mutually non-adjacent) in the polytope $P_l$ defined by the inequalities $I_l$.
 \item For each vertex $x$ of each $P_{l+1}$ the following information: either the fact that  $x$ is still a vertex in $P_l$ or the label of the unique facet inequality of $P_l$ that is violated at $x$.
 \end{enumerate}
 We call $k$ the \emph{depth} of the hierarchy and $|I_k|$ the \emph{core size}.
 \end{definition}

Observe that in part (3) uniqueness of the facet follows from the fact that the facets labelled by  $I_l\setminus I_{l+1}$ are independent in $P_l$.

\begin{lemma}
\label{lemma:DK}
Let $P=\{x\in \R^3: Ax \le b\}$ be a bounded $3$-polyhedron defined by $A\in \R^{m\times 3}$, $b\in \R^m$.
  Then, a Dobkin-Kirkpatrick hierarchy on $P$ of depth $O(m \log m)$ and base size $O(1)$  can be computed in time $O(\log m)$.
\end{lemma}

  \begin{proof} 
  First, it is well-known that the face poset of a $3$-polytope can be fully computed in the way we require in time $O(m\log m)$.
  
    Let $I=[m]$ be the row indices of $A$. 
    Let $I'$ be a subset of $[m]$ of size at most six and that defines a bounded polyhedron. This is, $\{x\in \R^3 : A_i x \leq b_i\ \  \forall i\in I'\}$ is a bounded set. 
 \footnote{Such an $I'$, of size at most $2d$, can be found in any facet-described $d$-polytope as follows: by inductive hypothesis assume that you know how to find such facets in polytopes of dimension smaller than $d$. To find them for $P$, start with any facet of $P$, say $I_1$ and solve the linear program $\min A^T_1 x$ on $P$. If the program has a unique minimum (a vertex) then let $I'$ be the original facet plus the $d$ containing that vertex. If the program is minimized at a face $F$ of dimension $0<d'<d$, then let $I'$ equal the original facet plus the $d-d'$ containing $F$ plus the at most $2d'$ that you can find by recursion. This gives $1+d+d'\le 2d$.
 
To find the facets, in the worst case you need to solve $d$ linear programs in dimension $\le d$, which can be done on $O(m)$ time (with a hidden constant depending on $d$).
}
    
    We now define the subsets $I=I_0, I_1, \dots, I_k$ of $I$ in the following recursive manner:
    Given $I_l$, we compute the face poset of the polyhedron $P_l$ defined by the rows of $Ax\leq b$ with indices in $I_l$.  
    We then compute a coloring of the facets of $P_l$ with at most $6$ colors, which can be done in linear time because the dual graph of $P_l$ is planar, so that the graph and all its subgraphs contain vertices of degree at most $5$.

We choose a color $C\subset I_l$ with 
\[
|C \cap (I_l\setminus I')| \ge \frac16 |I_l\setminus I'|
\]
and let $I_{l+1}= I_l \setminus C$.
    Eventually we reach an $I_k$ with $I_k=I'$, hence $|I_k| \leq 6 =O(1)$.
    Since each time we remove at least $1/6$th of the original inequalities (not in $I'$), $|I_l|\leq |I'|+\lceil {\left(\frac{5}{6}\right)}^l |I_0\setminus I'| \rceil$. Thus, we have at most $\log_{5/6}(m)+1$ steps in the hierarchy, so $k\in O(\log m)$. The whole computation needs time proportional to
    \begin{align*}
    \sum_{l=0}^k |I_l|\log(|I_l|) &\le \sum_{l=0}^k |I_l|\log(m) \le  \\
    &\le\log(m) \sum_{l=0}^{k} |I_l|\le \\
    &\le  \log(m) \left( 6+\sum_{l=0}^k (5/6)^l m\right) \le \\
    &\le  \log(m) \left( 6+6 m)\right) \le O(m\log m).
    \end{align*}

At each step we can easily identify which vertices appear and disappear, and what facet of $I_l$ is violated at each new vertex. When doing this each facet is only considered at one of the levels (the one in which it dissappears) and the total number of vertices in all layers is linear, since the sizes of the polytopes in the layers are bounded by a geometric sequence of ratio $5/6$.
\end{proof}

\begin{lemma}
Let $P$ be a facet-described polytope and $P'$ be a section of $P$ obtained by intersecting with a linear system of independent inequalities. Then, any Dobkin-Kirkpatrick hierarchy on $P$ is also a Dobkin-Kirkpatrick hierarchy on $P'$.
\end{lemma}

\begin{proof}
By induction on $\dim P - \dim P'$ it is enough to consider the case $\dim P - \dim P'=1$, so that $P'$ is obtained from $P$ by adding one inequality, that is, intersecting with a hyperplane $H$.

The first observation is that of a set of facets are mutually non-adjacent on $P_l$ then they are also mutually non-adjacent on $P'_l:= P_l\cap H$, so the stratification of $[m]$ in the hyerarchy of $P$ works also in $P'$. We need only to show how the hierarchy on $P$ allows to find the information of which facets remove which vertices on $P'$. For this, observe that a vertex $x$ of a $P'_l$ is an edge of the corresponding $P_l$. Let $u$ and $v$ be the end-points of that edge. Then, $x$ can only be eliminated by the facet of $P'_{l+1}$ that eliminates one (or both) of $u$ and $v$ in $P_{l+1}$, and this can be checked in logarithmic time (the time needed to find the vertices $u$ and $v$).
\end{proof}

\begin{theorem}
Let $P$ be a $d$-polytope with $m$ facets and suppose that we have a Dobkin-Kirkpatrick hyerarchy on $P$ of depth $k$ and base $O(1)$. Then, a linear program on $P$ can be solved in time $O(k^d \log m)$.
\end{theorem}

\begin{proof}
In order to solve the linear program with objective function $c^Tx$ we traverse the hierarchy in reverse. In the last polytope $P_k$ we need constant time since it has at most $O(1)$ facets. Once we have the maximizer $x_{l+1}^*$ in $P_{l+1}$ we find the maximizer in $P_l$ as follows: if $x_{l+1}^*$ is in $P_l$ (that is, if it satisfies the inequalities with indices in $I_l\setminus I_{l+1}$) then we set $x_{l}^*=x_{l+1}^*$. 
    
    If $x_{l+1}^*$ is not in $P_l$ then by construction of the hierarchy, there is a unique inequality in $I_l\setminus I_{l+1}$ violated by $x_{l+1}^*$ (this is because no two facets indexed by $I_l\setminus I_{l+1}$ are adjacent in $P_l$). We solve the linear program on that facet to find 
    $x_{l+1}^*$. By inductive hypothesis this step requires $O(k^{d-1} \log m)$, and we need to do this at most $k$ times.
 \end{proof}

\begin{corollary}
\label{coro:DK}
  Any facet described $3$-polytope with $m$ facets can be preprocessed in time $O(m \log m)$ so that linear programs on $P$ can be solved in time $O(\log^4 m)$ and linear programs on planar sections of it in 
time $O(\log^3 m)$
\end{corollary}

\section{Proof of Theorem~\ref{thm:main}}

  Without loss of generality let us assume $||A_i|| =1$ for every $i$. We also  assume that the given description of $P$ is irredundant (every row of $A$ is a facet), which we can check with a (dual) convex hull computation in $O(m \log m)$ time.

  We want to compute the value $\rho>0$ for which
 \pacoc{ Not sure if I want to include rho=0 or not. The original does not but I find more convenient including it as I need to study the function near 0 to show it's positive. It doesn't matter as the solution will never be rho=0 but we have to think which is more stylish.}
  \[ \width(P^{\rho})=2n\rho.\]

  Since $P^{\rho} = \interior_{\rho} (P) + B(0,\rho)$, we have that 
  \[
  \width(P^{\rho}) = \width(\interior_{\rho} (P) ) +2\rho.
  \]
  Hence, by definition of width, the $\rho$ and $v$ we are looking for must satisfy:
  \[ \min_{v \in \Sph^2} \left( \width_{v} \left( \interior_{\rho}(P) \right) - 2(n-1)\rho \right) = 0. \]

 The direction minimizing width in the polygon $ \interior_{\rho}(P)$ is normal to an edge of $ \interior_{\rho}(P)$,\footnote{A version of this is true in any dimension: by the Karush-Kuhn-Tucker conditions, the $v$ minimizing width in any polytope must be the common normal to two faces of $P$ with sum of dimensions $\ge d-1$.} hence to an edge of $P$.
  Thus, we do not need to check for all $v$, only those normal to edges of $P$. We use the outwards normals without loss of generality; that is, $v$ must be a row of $A$.

  Now let $r>0$ be the inradius of $P$; observe that $0<\rho<r$. For each $i\in [m]$, let  $f_i: [0,r] \rightarrow \R$ be defined as
  \[ f_i(t) = \width_{A_i} \left(\interior_{t}(P) \right) - 2(n-1)t ,\]
  so that the equation that characterizes $\rho$ is:
    \begin{equation}
 \min_{i\in [m]} f_i(\rho) = 0. 
 \label{eq:rho}
\end{equation}

  Each $f_i$ is well defined (as $\interior_{t} (P)$ is not empty for $0\le t < r$), continuous, piece-wise linear, and monotonically decreasing. 
  At $t=0$ every $f_i$ is positive and at $t=r$ some $f_i$ is negative because the width of $\interior_r (P)$ is $0$ in some direction.

  Then, \eqref{eq:rho} implies that $\rho$ is exactly:
  \[ \rho = \min \{ 0<t<r : \exists i\in [m] : f_i(t) = 0 \} .\]
  Indeed, some $f_i$ is guaranteed to have a root by continuity, and $\rho$ is a root for some $f_i$. If $\rho$ were not the minimum root, then some other root is smaller and by the $f_i$ being strictly decreasing, some $f_i$ is negative at $\rho$.

  So, in order to find $\rho$ we need only to compute the minimum of the roots of the $f_i$. This is not trivial, since the definition of each $f_i$ is quite implicit.
  However we need not verify all of them. For each $i\in [m]$, let $M_i$ be the maximum $t$ such that $A_i x \leq b-t$ still defines an edge of $\interior_t (P)$. Then, for $t>M_i$ the minimum width of $\interior_t(P)$ cannot be attained at the direction $A_i$, since it needs to be attained at the normal to an edge of $\interior_t (P)$.
  Thus,
  \[ \rho = \min \{ 0<t<r : \exists i\in [m] : f_i(t) = 0, t_i \le M_i   \} .\]
  Equivalently, by continuity and monotonicity,
  \[ \rho = \min \{ 0<t<r : \exists i\in [m] : f_i(t) = 0, f_i(M_i)\leq 0 \} .\]

We claim that (after preprocessing), we can compute each $M_i$ in polylogarithmic time. For this, consider the following three-dimensional polytope that we call the \emph{dome} of $P$:
\[
\widetilde P:=\{(x,y,t) \in \R^3 : Ax\leq b-t, t\geq 0 \}.
\]
That is, $\widetilde P \cap \{t=0\}$ equals $P$ and, apart from the horizontal facet, $\widetilde P$ has a facet with normal $(A_i,1)\in \R^3$ for each $i\in [m]$. Assume that we have preprocessed $\widetilde P$ as required in Corollary~\ref{coro:DK}.
The value $M_i$ equals the maximum of the $t$ coordinate of the $i$-th facet of $\widetilde P$ and, by the corollary,  it can be computed in time $O(\log^3 m)$.

Thus, we can compute all the $M_i$ in time $O(m\log^3 m)$. Once this is done 
we evaluate all $f_i(M_i)$ also in total time $O(m\log^3 m)$ by solving the linear programs with objective functions $A_i$ in the horizontal slices  
$P_i := \widetilde P \cap \{t=M_i\}$ of the dome.

  The motivation for adding the restrictions $t_i \le M_i$ is that in this range the functions $f_i$ are easier to compute. Recall that 
    \[ 
  f_i(t) = \max_{x: Ax\leq b-t} A_i^T x - \min_{x: Ax \leq b-t} A_i^T x - (2n-2)t .
  \]
Now,   in the range $0\le t \le M_i$ we have
  \[
   \max_{x: Ax\leq b-t} A_i^T x = b_i - t,
   \]
so we can rewrite
    \[ 
  f_i(t) = b_i - \min_{x: Ax \leq b-t} A_i^T x - (2n-1)t = \max_{x:Ax \leq b-t} \left(b_i-A_i^T x - (2n-1)t \right).
  \]

  We want to solve $f_i(t) =0$, for $0\leq t \leq M_i$ and $f_i(M_i)<0$. Equivalently, we want the unique $t$ such that:
  \[ 0 = \max_{x:Ax \leq b-t} \left(b_i-A_i^T x - (2n-1)t \right). \]

  This is the same as finding:
  \[ 
  \arg \quad \max_{t:0\leq t \leq M_i} \quad 
  \max_{\substack{x:Ax\leq b-t \\ A_i^Tx  \geq b_i- (2n-1)t}} \quad (b_i - A_i^Tx -(2n-1)t). 
  \]

This is a linear program on the dome, except we have an extra constraint $A_i^Tx  \geq b_i- (2n-1)t$. 
In order to solve it we solve it first without the constraint. 
If the optimum satisfies the extra constraint we are done, and if not the optimum we want is obtained solving the linear program in the section $\widetilde P \cap \{A_i^Tx  \geq b_i- (2n-1)t\}$. 
So, this linear program is solved in time $O(\log^4 m)$.
  The minimum of the solutions of these programs for the different choices of $i$ is the value of 
$\rho$ we are looking for.



\end{document}